\documentclass[reqno]{amsart}

\usepackage[margin=3.5cm]{geometry}
\usepackage[T1]{fontenc}
\usepackage{amsmath,amssymb,amsxtra,amscd}
\usepackage{amsthm}
\usepackage[all]{xy}
\usepackage{pStrongSegalNotation}

\newtheorem{theorem}{Theorem}[section]
\newtheorem{proposition}[theorem]{Proposition}
\newtheorem{lemma}[theorem]{Lemma}
\newtheorem{definition}[theorem]{Definition}

\newtheorem{corollary}[theorem]{Corollary}

\newtheorem{maintheorem}{Theorem}

\newtheorem*{question*}{Question}

\theoremstyle{definition}

\newcommand{\bm}[1]{\mbox{\boldmath $#1$}}

\newcommand{\invlim}{\ensuremath{\underleftarrow{\operatorname{lim}}}}

\newcommand{\sma}{\wedge}

\begin{document}

\title[Completion of $G$-spectra and stable maps between classifying spaces]{Completion of \bm{G}-spectra and stable maps between classifying spaces}


\author{K\'ari Ragnarsson}
\address{Department of Mathematical Sciences\\
         Depaul University\\ 
         Chicago, IL \\ 
         USA}
\email{kragnars@math.depaul.edu}
\date{March 17, 2011}

\begin{abstract} We prove structural theorems  for computing the completion of a $G$-spectrum at the augmentation ideal of the Burnside ring of a finite group $G$. First we show that a $G$-spectrum can be replaced by a spectrum obtained by allowing only isotropy groups of prime power order without changing the homotopy type of the completion. We then show that this completion can be computed as a homotopy colimit of completions of spectra obtained by further restricting isotropy to one prime at a time, and that these completions can be computed in terms of completion at a prime.
 
As an application, we show that the spectrum of stable maps from $BG$ to the classifying space of a compact Lie group $K$ splits non-equivariantly as a wedge sum of $p$-completed suspension spectra of classifying spaces of certain subquotients of $G \times K$. In particular this describes the  dual of $BG$. 
\end{abstract}

\maketitle

\section{Introduction}
Completions of $G$-spectra at ideals of the Burnside ring were introduced by Greenlees--May in \cite{GreenleesMay:CompBurnside} as a means to give a topological flavour to completion theorems in stable equivariant homotopy theory. We say that the completion theorem holds for a $G$-spectrum $X$ if the natural map
 \[ X \cong F(pt_+ , X) \longrightarrow F(EG_+,X) \, ,  \]
induced by the projection $EG \to pt$ is an $I(G)$-adic completion map, where $I(G)$ is the augmentation ideal in the Burnside ring $A(G)$. 
(In \cite{GreenleesMay:CompBurnside} $G$ can be a compact Lie group, but we consider only finite groups in this paper.) The subscript $+$ is used here to denote an added disjoint basepoint.

Completion theorems are especially helpful in the case where $X$ is $G$-split, in the sense of \cite{MayMcClure:Segal}, for then one has a non-equivariant equivalence $F(EG_+,X)^G \simeq F(BG_+,i^*X),$ where $i^*X$ is the underlying non-equivariant spectrum of $X$. Thus, if the completion theorem holds for a $G$-split $G$-spectrum $X$, then the natural map
 \[  X^G \longrightarrow F(BG_+,i^*X) \]
is an $I(G)$-adic completion map. Taking homotopy groups, one deduces that the natural map
 \[ X^*_G(pt_+) \longrightarrow X^*(BG_+) \]
is an $I(G)$-adic completion map. This elucidates the importance of completion theorems: one is often interested in $X^*(BG_+)$, and $X^*_G(pt_+)$ is usually easier to compute. 

The original, and most important, examples of completion theorems are the Atiyah--Segal completion theorem \cite{Atiyah:CharactersAndCohomology,AtiyahSegal:Completion}, where $X$ is real or complex periodic equivariant $K$-theory, and the Segal conjecture, proved by Carlsson in \cite{Carlsson:SegalConjecture}, where $X$ is the equivariant sphere spectrum $\SphereSpectrum_G$. These were each originally proved as cohomological statements and lifted to the spectrum level in \cite{GreenleesMay:CompBurnside}. The Atiyah--Segal completion theorem gives, in the complex case, an isomorphism 
 \[ \hat{R}[G] \overset{\cong}{\longrightarrow} KU^0(BG) \, \]
where $\hat{R}[G]$ is the completion of the complex representation ring at its augmentation ideal\footnote{$A(G)$ acts on $R[G]$ in an obvious way, and it is not hard to show that completion at $I(G)$ is equivalent to completion at the augmentation ideal of $R[G]$.}, and a similar statement in the real case. This inspired Segal to conjecture an isomorphism
 \[ A(G)^\wedge_{I(G)} \overset{\cong}{\longrightarrow} \pi^0(BG_+) \, , \]
which he later strengthened to a stronger from, involving higher cohomotopy groups. It is the stronger form which Carlsson proved.

While completion theorems are very descriptive, their utility is somewhat restricted by the fact that $I(G)$-adic completion has so far been very difficult to compute in practice. A notable exception is when $G$ is a $p$-group, for in this case results of May--McClure from \cite{MayMcClure:Segal} on completion of Mackey functors suggest that the $p$-adic completion of $X$ factors through the $I(G)$-adic completion, and the difference is controlled by the underlying non-equivariant spectrum. In this paper we make this spectrum-level analogue of the May--McClure result precise, and generalize it to an arbitrary finite group $G$, at the cost of restricting to $G$-spectra whose isotropy groups are $p$-groups (see Theorem \ref{mthm:X[F]Comp}) below. 
 
More generally, we present an effective approach for computing the $I(G)$-adic completion of an arbitrary $G$-spectrum $X$ by restricting isotropy to $p$-groups, one prime at a time. The key to this is a result which allows us to replace $X$ by a spectrum whose $I(G)$-adic completion is easier to handle. For a $G$-spectrum $X$, we form a $G$-spectrum $X[\F_\mathbf{P}]$ which one can informally think of as the subspectrum of $X$ whose isotropy groups are of prime power order; formally $X[\F_\mathbf{P}] \defeq {E\F_\mathbf{P}}_+ \sma X$, where $E\F_\mathbf{P}$ is the universal space for the family $\F_\mathbf{P}$ of subgroups in $G$ of prime power order (for any prime). The canonical map $E\F_\mathbf{P} \to pt$ induces a map $X[\F_\mathbf{P}] \to X$, and we have the following result.
 
\begin{maintheorem}[Replacement theorem] \label{mthm:Replacement}
For a bounded-below $G$-spectrum $X$, the natural map $X[\F_\mathbf{P}] \to X$ induces a weak equivalence upon $I(G)$-adic completion.
\end{maintheorem}

Of course, when $X$ is a $G$-CW spectrum, the replacement theorem gives an actual homotopy equivalence $X[\F_\mathbf{P}]^\wedge_{I(G)} \simeq X^\wedge_{I(G)}$. 

The next step in our process is to decompose $X[\F_\mathbf{P}]^\wedge_{I(G)}$ into component spectra $X[\F_p]^\wedge_{I(G)}$, for different primes $p$, where $\F_p$ is the family of $p$-subgroups in $G$. This decomposition is not quite a wedge sum splitting, as each component $X[\F_p]^\wedge_{I(G)}$ contains the spectrum $X[\F_1]^\wedge_{I(G)}$, where $\F_1$ is the family containing only the trivial subgroup. Using a homotopy colimit to keep track of this, we obtain the following result.
\begin{maintheorem}[Decomposition theorem] \label{mthm:Decomposition}
For a bounded-below $G$-spectrum $X$, $X[\F_\mathbf{P}]^\wedge_{I(G)}$ is weakly equivalent to the homotopy colimit of the diagram consisting of the natural maps $X[\F_1]^\wedge_{I(G)} \to X[\F_p]^\wedge_{I(G)}$ for all primes $p$.
\end{maintheorem}

It clearly suffices to take the homotopy colimit over primes that divide the order of $G$, since the map $X[\F_1]^\wedge_{I(G)} \to X[\F_p]^\wedge_{I(G)}$ is a weak equivalence for other primes.  However, it is theoretically useful to consider the diagram ranging over all primes, as this allows us to deduce naturality in $G$ in certain circumstances.

The completions of the component spectra can be computed as follows.
\begin{maintheorem} \label{mthm:X[F]Comp}
Let $X$ be a bounded-below $G$-spectrum.
\begin{enumerate}
\item[(a)] The completion map $X[\F_1] \to X[\F_1]^\wedge_{I(G)}$ is a weak equivalence.
\item[(b)] For a prime $p$, the $I(G)$-adic completion of the cofibration sequence
  \[ X[\F_1] \longrightarrow X[\F_p] \longrightarrow X[\F_p,\F_1] \]
is weakly equivalent to the cofibration sequence
  \[ X[\F_1] \longrightarrow X[\F_p]^\wedge_{I(G)} \longrightarrow \pComp{X[\F_p,\F_1]} \, .\]
\end{enumerate}
\end{maintheorem}
Theorem \ref{mthm:X[F]Comp} makes precise the loosely worded statement above that the $p$-adic completion of  $X[\F_p]$ factors through the $I(G)$-adic completion, and the difference between them lies only in the subspectrum $X[\F_1]$, which is unchanged by completion. When $G$ is a $p$-group one has $X[\F_p] \simeq X$, and so one obtains the promised spectrum-level analogue of the May--McClure result.

As an application of these results we can now give an explicit description of the spectrum of stable maps between classifying spaces. This description is a generalization of the author's work on homotopy classes of stable maps between classifying spaces in \cite{KR:pSegal}, and indeed this was the motivation for the work in this paper.

Even before its proof, Lewis--May--McClure showed in \cite{LewisMayMcClure:Segal} that the Segal conjecture implies the completion theorem for the suspension spectrum of $B(G,K)$, the classifying space of principal $K$-bundles in the category of $G$-spaces. Their result was for finite groups $K$, and this was extended to compact Lie groups in \cite{MaySnaithZelewski:FurtherGeneralizationOfSegal}. The spectrum $\Stable_G B(G,K)_+$ is $G$-split, and there results a non-equivariant homotopy equivalence
 \[ (\Stable_G B(G,K)_+^G)^\wedge_{I(G)} \simeq F(BG_+,\Stable BK_+) \, .\]
There is a non-equivariant splitting (\cite{LewisMayMcClure:Segal,tomDieck:TransformationGroups})
\[  \Stable B(G,K)_+^G \simeq \bigvee_{[H,\varphi]} \Stable BW(H,\varphi)_+, \]
where the wedge sum runs over conjugacy classes of pairs $(H,\varphi)$ consisting of a subgroup $H \leq G$ and a homomorphism $\varphi \colon H \to K$, and $W(H,\varphi)$ is a certain subquotient of $G\times K$ (see Section \ref{sec:StableMaps} for details). Hence $F(BG_+,\Stable BK_+)$ is equivalent to the $I(G)$-adic completion of this wedge sum of classifying spaces, and this is usually considered the best available description. 

The problem with this point of view is that, except when $G$ is a $p$-group, there has been no way to compute this completion; the action of $A(G)$ on the wedge sum is not even understood. This problem can now be overcome: using the Theorem \ref{mthm:Replacement} one can restrict the wedge sum to subgroups of prime power order, and using Theorems \ref{mthm:Decomposition} and \ref{mthm:X[F]Comp}, the completion of each remaining classifying space can be described in terms of completion at the relevant prime, leading to the following result.

\begin{maintheorem} \label{mthm:StableMaps}
For a finite group $G$ and a compact Lie group $K$, there is a homotopy equivalence
\[ F(BG_+, \Stable BK_+) 
    \simeq  \bigvee_{\substack{[H,\varphi] \\ H \in \F_\mathbf{P}}} 
    \Stable (BW(H,\varphi)^{\wedge}_{p(H)} )_+ \, , 
\] 
where, for $H \in \F_\mathbf{P}\setminus \F_1$, $p(H)$ is the prime that divides $|H|$, and $p(1) = 0$.
\end{maintheorem}
Applying the functor $\pi_0$, one independently recovers \cite[Theorem B]{KR:pSegal}. One can also focus on a single prime $p$ to obtain a wedge sum description of $F(\pComp{BG}{}_+, \Stable BK_+)$ (see Corollary \ref{cor:CompAsWedge}), and applying $\pi_0$ to that description one independently recovers \cite[Theorem A]{KR:pSegal}. Taking $K =1$ in these results one obtains a description of the mapping duals of $BG_+$ and $\pComp{BG}_+$ (see Corollary \ref{cor:MappingDuals}.

\subsection*{Outline} Following this introduction are three preliminary sections with mostly background material. In Section \ref{sec:FamClSp} we review families of subgroups and their universal spaces. Section \ref{sec:ESHT} contains a very brief overview of equivariant stable homotopy theory, as well as a proof of the known but unpublished result that weak equivalences can be detected on geometric fixed-point spectra (Proposition \ref{prop:WeakGeometric}). Completion at ideals of the Burnside ring is reviewed in Section \ref{sec:Completions}. The real work begins in Section \ref{sec:Replacement}, where we introduce the restricted isotropy spectra $X[\F]$, investigate their basic properties and prove Theorem \ref{mthm:Replacement}. Theorems \ref{mthm:Decomposition} and \ref{mthm:X[F]Comp} are then proved in Section \ref{sec:Decomposition}. Finally, the application to stable maps between classifying spaces is treated in Section \ref{sec:StableMaps}.

\subsection*{Acknowledgments} The author is grateful to John Greenlees and Peter May for their permission to reproduce material from their overview article \cite{GreenleesMay:EqStHoTh} in Section \ref{sec:ESHT}, and for their help regarding the proof and historical context of Proposition \ref{prop:WeakGeometric}. The author also thanks Tony Elmendorf for an explanation of his functor realizing systems of fixed points in \cite{Elmendorf:SystemsOfFixedPoints}, which allows the functorial construction of the restricted isotropy spectra. The method of proof is heavily dependent on the results of Greenlees from \cite{Greenlees:EqFuncDuals}, and the author hopes this goes some way toward answering the question implicitly asked in the MathSciNet review of that paper.

\section{Families of subgroups and classifying spaces} \label{sec:FamClSp}
In this section, and throughout the paper, fix a finite group $G$. A \emph{family} $\F$ of subgroups of $G$ is a set of subgroups of $G$ that is closed under conjugacy in $G$ and taking subgroups. The particular families that will be used in this paper are $\F_1$, the family consisting solely of the trivial subgroup; $\F_p$, the family of $p$-subgroups for a prime $p$; $\F_\mathbf{P} = \bigcup_p \F_p$, the family of subgroups of prime power order; and $\F_{all}$, the family of all subgroups. 

Given a family $\F$, a $G$-space $X$ is said to be an \emph{$\F$-space} if the isotropy group of each point in $X$ belongs to $\F$. There is a distinguished $\F$-space $E\F$ with the universal property that each $\F$-space $X$ admits a unique homotopy class of maps $X \to E\F$. The space $E\F$ can be characterized by the fact that, for a subgroup $H \leq G$, the fixed-point space $E\F^H$ is contractible if $H \in \F$, and empty if $H \notin \F$. Of course, these properties only determine $E\F$ up to homotopy, and to be explicit $E\F$ will here be taken to mean the space constructed by Elmendorf in \cite{Elmendorf:SystemsOfFixedPoints}. Elmendorf describes a functor that, given a system of fixed points for subgroups of $G$, uses a categorical bar construction to construct a $G$-space realizing that system. The following lemma is a consequence of the construction.
\begin{lemma}\label{lem:EFnice}
For a family $\F$, the space $E\F$ is of finite type, and an inclusion of families $\F_a \hookrightarrow \F_b$ induces an inclusion of $G$-CW complexes $E\F_a \to E\F_b$, in particular a cofibration.
\end{lemma}

The homotopy description of universal implies that $E\F_1 \simeq EG$ and $E\F_{all} \simeq pt$ as $G$-spaces. The spaces $E\F_p$ and $\F_\mathbf{P}$ are difficult to get a direct handle on, but they do relate to each other via a colimit description, which we will use throughout the paper. 

\begin{definition}\label{def:Diagram}
Let $\calD$ be the category whose objects are the families $\F_1$ and $\F_p$, for all primes $p$, and whose morphisms are the inclusions $\F_1 \hookrightarrow \F_p$.
\end{definition}

\begin{lemma} \label{lem:EFColimit}
The natural map  
\[ \colim_{\F \in \calD} E\F \longrightarrow E\F_\mathbf{P}   \]
is a homeomorphism.
\end{lemma}
\begin{proof} The system of fixed points for $\F_\mathbf{P}$ is the colimit of the systems of fixed points of the families in $\calD$. The point is now that the categorical bar construction is a big colimit construction, and colimits commute.
\end{proof}

Although a direct proof is easy in this case, it is fitting to attribute Lemma \ref{lem:EFColimit} to Piacenza, who in \cite{Piacenza:HomotopyTheoryOfDiagrams} showed that Elmendorf's functor in \cite{Elmendorf:SystemsOfFixedPoints} is a Quillen equivalence from the systems of fixed points to $G$-spaces.

The colimit in Lemma \ref{lem:EFColimit} is taken over the diagram
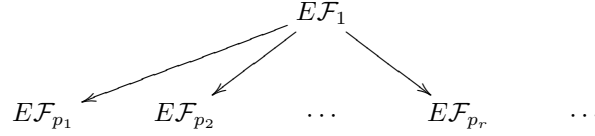
\[ 
\xymatrix{                 &              &E\F_1 \ar[1,-2] \ar[1,-1] \ar[1,1]      &    &  \\
           E\F_{p_1}       &E\F_{p_2}     &\cdots     &E\F_{p_r}  &\cdots 
        }
\]
where $p_1, p_2, \dots$ is an enumeration of the primes. Thus $E{\F_\mathbf{P}}$ is very close to being the coproduct of the spaces $E\F_p$: the difference is that each $E\F_p$ contains $E\F_1$ as a subspace, and the colimit construction is needed to keep track of this. Alternatively, one can describe $E\F_\mathbf{P}$ as a coproduct in the category of $G$-spaces under $E\F_1$, but there does not seem to be much to gain from this. 

We also observe that $E\F_p = E\F_1$ for primes $p$ that do not divide the order of $G$, and so it suffices to take the colimit in Lemma \ref{lem:EFColimit} over primes that divide the order of $G$. This is nice since one then obtains a finite colimit. For theoretical reasons it is nevertheless convenient to include all primes, as that makes the construction natural in $G$.

\section{Equivariant stable homotopy theory} \label{sec:ESHT}
The subject of equivariant stable homotopy theory is much too vast to cover in any detail in this article.  An excellent overview of the subject is provided in \cite{GreenleesMay:EqStHoTh}, and a detailed account can be found in \cite{LewisMaySteinbergerMcClure:EqStHoTh}. In this section we paraphrase passages from \cite{GreenleesMay:EqStHoTh} with the aim that a reader familiar with the framework for non-equivariant stable homotopy in \cite{EKMM} can follow the discussion in the paper, without going into much technical detail. The author is grateful to John Greenlees and Peter May for the permission to reproduce their work here.
In Proposition \ref{prop:WeakGeometric} we also fill a gap in the literature by presenting a proof that actual fixed points and geometric fixed points give rise to the same notion of weak equivalence.

Fix a \emph{complete $G$-universe $U$}. This means $U$ is a real inner product space on which $G$ acts by linear isometries, and $U$ decomposes as a sum of countably many copies of each irreducible representations of $G$. A \emph{$G$-spectrum} $E$ (indexed on $U$) is then a collection of based $G$-spaces $E_V$, one for each finite-dimensional $G$-subspace $V \subseteq U$, equipped with compatible $G$-homeomorphisms
  \[ E_V \xrightarrow{\cong} \Omega^{W-V}E_{W-V} = \Map(S^{W-V},E_{W-V}),  \]
one for each subspace inclusion $V \subset W$, where $W-V$ is the orthogonal complement of $V$ in $W$, and $S^{W-V}$ is the one-point compactification of $W-V$. (Observe that $S^{W-V}$ is a sphere of dimension $\dim W - \dim V$, with $G$-action inherited from $W$.) A map $E \to E'$ of $G$-spectra is a collection of $G$-space maps $E_V \to E'_V$, compatible with the structure maps. Equivariant suspension spectra ($\Stable_G X$), smash products ($X \wedge Y$), and function spectra ($F(X,Y)$) are defined, and behave, much as in the non-equivariant setting. Consequently one obtains mapping cones, mapping cylinders and homotopies in the standard way. For $G$-spectra $E$ and $F$, we write $[E,F]_G$ for the set of homotopy classes of maps $E \to F$. 

Every $G$-spectrum has an \emph{underlying non-equivariant spectrum}. For a $G$-spectrum $E$, this is obtained by first pulling $E$ back along the inclusion of fixed points $i \colon U^G \hookrightarrow U$ to obtain the \emph{naive $G$-spectrum} $i^*E$ indexed on the $G$-fixed universe $U^G$, and then forgetting the $G$-action. A \emph{non-equivariant homotopy equivalence} then means a homotopy equivalence between underlying non-equivariant spectra.

In equivariant homotopy theory, $G$-orbits play the role of points, and therefore one must consider the generalized sphere spectra $G/H_+ \sma \SphereSpectrum_G^n$, where $\SphereSpectrum_G^n$ is the $n$-fold  suspension of the \emph{$G$-equivariant sphere spectrum} $\SphereSpectrum_G \defeq \Stable_G S^0$. Consequently the \emph{homotopy groups} of a $G$-spectrum $E$ are defined by setting
 \[ \pi^H_n(E) \defeq [G/H_+ \sma \SphereSpectrum_G^n, E] \,  \]
for $H \leq G$ and $n \in \Z$. This gives rise to a Mackey functor 
\[\pi_n(E) \colon H \mapsto  \pi^{H}_n(E) \, .\] 
A spectrum $E$ is \emph{bounded below} if there exists an integer $N$ such that $\pi_n(E) = 0$ for $n < N$.

A map between $G$-spectra is a \emph{weak equivalence} if  it induces an isomorphism on all homotopy groups. \emph{Cofibrations} and \emph{fibrations} are defined via the homotopy extension and homotopy lifting properties, respectively (here we ignore model structures and these terms are meant in the classical sense). As in the non-equivariant world, cofibration sequences coincide with fibration sequences up to sign and weak homotopy. It follows that fibre and cofibre sequences of $G$-spectra both give rise to a long exact sequence of homotopy groups, regarded as Mackey functors or (equivalently) one orbit at a time.

Again, since $G$-orbits play the role of points, it is appropriate to regard the $G$-spectra \mbox{$G/H_+ \sma C\SphereSpectrum_G^n$}, where $C\SphereSpectrum_G^n = I \sma \SphereSpectrum_G^n $ is the cone on $\SphereSpectrum_G^n$, as generalized cells. \emph{$G$-CW spectra} are then constructed by attaching such cells along the boundary sphere $G$-spectra  $G/H_+ \sma \SphereSpectrum_G^n$. The standard results for CW-spectra carry over to the equivariant world: We have an \emph{equivariant Whitehead theorem} (a weak equivalence between $G$-CW spectra is a homotopy equivalence), cellular approximation of maps and homotopies, and approximation by $G$-CW spectra. The last result provides a functor $\Gamma$ that assigns a $G$-CW spectrum $\Gamma E$ to a $G$-spectrum $E$, and comes equipped with a natural weak equivalence $\gamma \colon \Gamma E \to E$. We say that a $G$-spectrum is of \emph{finite type} if it is weakly equivalent to a $G$-CW spectrum with finitely many cells in each dimension.

For a subgroup $H \leq G$, one can regard the complete $G$-universe $U$ as a complete $H$-universe, and this gives rise to a forgetful functor from $G$-spectra to $H$-spectra. We denote this functor by $i_H^*$, or $i^*$ when there is no danger of confusion. The restriction functors preserve weak equivalences and cofibre sequences, and have left and right adjoint functors, which will not be discussed here.

Fixed points of $G$-spectra have to be handled with some care. It is tempting to form fixed-point spectra by taking fixed points at the space level. The problem is that this is not compatible with the structure maps in the spectrum, as $\Map(S^{W-V},(E_{W-V})^G)$ is generally not homeomorphic to  $\Map(S^{W-V},E_{W-V})^G$ when $G$ acts non-trivially on $W$. One way around this problem is to restrict indexing to subspaces with trivial $G$-action and consider the naive $G$-spectrum $i^*E$. Taking space-level fixed points of $i^*E$ one gets a (non-equivariant) spectrum $(i^*E)^G$, indexed on on the $G$-fixed universe $U^G$, and the fixed-point spectrum $E^G$ is defined by $E^G \defeq (i^*E)^G$. The fixed-point-spectrum functor satisfies the expected adjunction relations, and is used to define equivariant homology and cohomology groups. However it does not mirror the unstable fixed-point functor (or preserve smash products) as information about representations of $G$ is built into $\Stable_G X$, even when $X$ has trivial $G$-action. Instead, one has the homotopy equivalence
\begin{equation} 
  (\Stable_G X_+)^G \simeq \bigvee_{(H)} \Stable(EW_G(H)_+ \wedge_{W_G(H)} X^H),
\end{equation}
originally due to tom Dieck \cite{tomDieck:TransformationGroups}, where the wedge sum runs over conjugacy classes of subgroups of $H$, and $W_G(H) = N_G(H)/H$ is the Weyl group of $H$ in $G$.

An alternative approach to fixed-point spectra is taking space-level fixed points of a $G$-spectrum, and then spectrifying the resulting prespectrum. This construction induces the \emph{geometric fixed-point spectrum} functor $\Phi^G$, so named because it satisfies $\Phi^G(E \wedge E') \simeq \Phi^G(E) \wedge \Phi^G(E')$ for $G$-spectra $E, E'$, and $\Phi^G(\Stable X) \simeq \Stable(X^G)$ for a based $G$-space $X$, in line with geometric intuition. It does not, however, satisfy the adjunction relation one demands of a fixed-point functor, and so plays a secondary, but important, role in equivariant stable homotopy theory. Formally, the geometric fixed-point spectrum is defined by a separation of isotropies. That is, one lets $\calP$ be the family of proper subgroups in $G$, and defines $\widetilde{E}\calP$ to be the cofibre of the natural map $E\calP_+ \to {E\F_{all}}_+ \simeq S^0$. For a $G$-spectrum $E$, one can informally think of the smash product $\widetilde{E} \calP \sma E$ as removing orbits whose isotropy groups are proper subgroups of $G$. Geometric intuition dictates that such orbits should have no $G$-fixed points (although this is false in equivariant stable homotopy theory), and this motivates the definition of the geometric fixed-point spectrum as  
 \[ \Phi^G(E) \defeq (\widetilde{E} \calP \sma E )^G \, .\]

For a subgroup $H \leq G$, suitable adaptations of the above constructions yield a fixed-point functor $(-)^H$ and a geometric fixed-point functor $\Phi^H$, both taking values in $W_G(H)$-spectra. We will only consider fixed-point spectra non-equivariantly, so it suffices to know that there are non-equivariant equivalences $E^H \simeq (i^*E)^H$ and $\Phi^H(E) \simeq \Phi^H(i^*E)$, where $i^*E$ is the restriction of $E$ to an $H$-spectrum. (A more elaborate construction is needed to obtain the $W_G(H)$-spectrum structure.) The $H$-fixed-point spectra of a $G$-spectrum $E$ satisfy the adjunction
 \[ \pi^H_*(E) \cong \pi_*(E^H)\, .\]
It follows immediately that a map between $G$-spectra is a weak equivalence if and only if it induces non-equivariant weak equivalences on all fixed-point spectra. It is known to experts that this remains true if one looks at geometric fixed-point spectra instead of actual fixed-point spectra, although no proof seems to appear in the literature. We remedy this in the following proposition, which was originally proved, but not published, by Gaunce Lewis, and also independently discovered by John Greenlees. An outline of the proof presented here was communicated to the author by Peter May.  
 
\begin{proposition} \label{prop:WeakGeometric}
 For a map \mbox{$f \colon X \to Y$} between $G$-spectra, the following are equivalent.
  \begin{itemize}
    \item[
 		(i)] $f$ is a weak equivalence.
 		\item[(ii)] For each $H \leq G$, the map $f^H$ is a non-equivariant weak  equivalence.
 		\item[(iii)] For each $H \leq G$, the map $\Phi^H(f)$ is a non-equivariant weak  equivalence.
	\end{itemize}
\end{proposition}
\begin{proof}
The equivalence between (i) and (ii) has already been established. The implication (i) $\Rightarrow$ (iii) is purely formal: A weak equivalence is an isomorphism in the homotopy category of $G$-spectra (obtained by formally inverting weak equivalences). It follows that the restriction to $H$-spectra $i^*f \colon i^*X \to i^*Y$ is an isomorphism in the homotopy category of $H$-spectra. This remains true after smashing with $\widetilde{E}\calP$, where $\calP$ is the family of proper subgroups of $H$. Taking $H$-fixed points, we get an isomorphism in the homotopy category of spectra,
\[ ( \widetilde{E}\calP \sma {i^*X} )^H \xrightarrow{(1 \sma {i^*f} )^H} (\widetilde{E}\calP \sma {i^*Y})^H, \]
which is the same as a non-equivariant weak equivalence $\Phi^H(X) \xrightarrow{\Phi^H(f)} \Phi^H(Y)$. 

It now suffices to prove the implication (iii) $\Rightarrow$ (ii). By taking $G$-CW approximations if necessary, the implication (i) $\Rightarrow$ (iii) allows us to assume, without loss of generality, that $X$ and $Y$ are $G$-CW spectra. We now proceed by induction on $G$, noting that the claim is trivially true when $G$ is the trivial group. We may then inductively assume that the claim is true for all proper subgroups $H < G$. In particular this means that restricting to $H$-spectra yields non-equivariant weak equivalences 
 \[ (i^*_H X)^H \xrightarrow{(i^*_H f)^H} (i^*_H Y)^H,  \]
which non-equivariantly is the same as weak equivalences 
 \[ X^H \xrightarrow{ f^H}  Y^H.  \] 
Thus it remains only to prove that $f^G$ is a weak equivalence. Letting $Z$ be the cofibre of the map $f \colon X \to Y$, this is equivalent to showing that $Z^G$ is weakly contractible. 

Consider the map $Z \cong S^0 \sma Z  \xrightarrow{\iota \sma 1}  \widetilde{E}\calP \sma Z$ where $\iota$ is the map in the cofibre sequence $E\calP_+ \to S^0 \xrightarrow{\iota} \widetilde{E}\calP$. Since $Z^H$ is weakly contractible for every $H \in \calP$, \cite[Proposition 9.2]{LewisMaySteinbergerMcClure:EqStHoTh} implies that $\iota \sma 1$ is a $G$-equivalence, and in particular $Z^G \simeq (\widetilde{E}\calP \sma Z)^G = \Phi^G(Z)$. But $\Phi^G(Z)$ is the cofibre of the map $\Phi^G(X) \xrightarrow{\Phi^G(f)} \Phi^G(Y), $ which is a non-equivariant weak equivalence and hence $\Phi^G(Z)$ is weakly contractible. It follows that $Z^G$ is weakly contractible, completing the induction and the proof.
\end{proof}
 
Finally, a $G$-spectrum $E$ is \emph{$G$-split} if the natural inclusion of $E^G$ into the underlying non-equivariant spectrum of $E$ splits up to homotopy. In other words, $E$ is $G$-split if there is a non-equivariant map $i^*E \to E^G$, whose composite with the non-equivariant inclusion $E^G \hookrightarrow i^*E$ is homotopic to the identity. $G$-split $G$-spectra play an important role in equivariant stable homotopy theory because of their interaction with \emph{free $G$-spectra}, meaning $G$-spectra that are equivalent to $G$-CW spectra built out of free $G$-cells. To keep things simple, we explain this interaction only for free $G$-CW complexes, in which case the result is due to Kosniowski \cite{Kosniowski:EquivariantCohomologyStableCohomotopy}, although a proof is also given in \cite{MayMcClure:Segal}. 

\begin{theorem} \cite{Kosniowski:EquivariantCohomologyStableCohomotopy,MayMcClure:Segal} \label{thm:Split}
If $E$ is a $G$-split $G$-spectrum and $X$ is a free $G$-CW complex, then there is a natural isomorphism
\[ E^n_G(X) \xrightarrow{\cong} E^n(X/G), \]
where $E^*$ denotes the cohomology theory represented by the underlying non-equivariant spectrum of $E$. 
\end{theorem}
 
A more general version of the theorem, where $X$ is allowed to be a free naive $G$-spectrum is proved in \cite{LewisMaySteinbergerMcClure:EqStHoTh}, but this will not be needed here. 
 
\section{The Burnside ring and completion theorems} \label{sec:Completions}
The \emph{Burnside ring} of a finite group $G$, denoted $A(G)$, is the Grothendieck group of isomorphism classes of finite left $G$-sets. Its additive and multiplicative structures are induced by disjoint union and cartesian product, respectively. As a $\Z$-module, the Burnside ring is generated by the transitive $G$-sets.  Thus, letting $[X]$ denote the isomorphism class of a $G$-set $X$, $A(G)$ has one basis element $[G/H]$ for each conjugacy class of subgroups $H \leq G$. 

\subsection{Algebraic completion} 

For an ideal $J \subseteq A(G)$ and an $A(G)$-module $M$, the $J$-adic completion of $M$ is defined by
\[ M^{\wedge}_J \defeq \invlim~ \left(M/JM \leftarrow M/J^2M \leftarrow M/J^3M \leftarrow \cdots \right) \]
As we are taking $G$ to be finite, $A(G)$ is Noetherian, so when $M$ is finitely generated the Artin--Rees Lemma applies and we have $M^{\wedge}_J \cong A(G)^{\wedge}_J \otimes_{A(G)} M$. Moreover, $A(G)^{\wedge}_J$ is flat over $A(G)$, and thus $J$-adic completion commutes with finite colimits for finitely generated $A(G)$-modules.

We will primarily be interested in completion at the augmentation ideal of $A(G)$, defined as follows. The \emph{augmentation} of $A(G)$ is the homomorphism $\epsilon$  
defined on isomorphism classes of $G$-sets by setting
 \[ \epsilon([X]) = |X| \, ,\]
where $|X|$ denotes the cardinality of $X$, and then extending to a $\Z$-linear homomorphism 
\[\epsilon\colon A(G) \to \Z \, .\] 
As usual, the \emph{augmentation ideal} $I = I(G)$ is defined as the kernel of the augmentation:
\[ I = I(G) \defeq{\ker{\epsilon}} \, . \]

\subsection{Cohomological completion theorems}
The main motivation for considering completion at the augmentation ideal is the Segal Conjecture, which stipulates that the natural map
 \[ A(G)^\wedge_I \xrightarrow{\cong} \pi^0(BG_+), \]
where $\pi^*$ denotes the cohomotopy functor, is an isomorphism. The conjecture was settled by Carlsson in \cite{Carlsson:SegalConjecture} where he bridged the gap between the results obtained in \cite{AdamsGunawardenaMiller,CarusoMayPriddy} and \cite{MayMcClure:Segal}. Carlsson actually proved the ``strong Segal conjecture'', describing also the higher cohomotopy groups of $BG_+$ as an $I$-adic completion. His theorem has since undergone a series of generalizations in
\cite{AdamsHaeberlyJackowskiMay:GeneralizationOfSegal,LewisMayMcClure:Segal,MaySnaithZelewski:FurtherGeneralizationOfSegal} to, amongst other things, give a description of the homotopy groups of a mapping spectrum $F(BG_+,\Stable {BK_+})$ for a finite group $G$ and a compact Lie group $K$ in terms of $I$-adic completions. This will be discussed in more detail in Section \ref{sec:StableMaps}.

To understand the strong Segal conjecture and its generalizations, one must first recall Segal's identification,
 \[ A(G) \cong [\SphereSpectrum_G,\SphereSpectrum_G]^G, \]
where $\SphereSpectrum_G$ denotes the $G$-equivariant sphere spectrum. (See \cite{tomDieck:TransformationGroups} for a proof and generalization to compact Lie groups.) For a $G$-spectrum $E$, the homotopy groups of the function spectrum $F(\SphereSpectrum_G,E)$ then receive an $A(G)$-module structure by pre-composition, and can hence be subjected to $I$-adic completion. We say that the \emph{completion theorem holds for $E$} if the natural map $E^*_G(\SphereSpectrum_G) \to E^*_G(EG_+)$, induced by the map $EG \to pt$,  is an $I$-adic completion map. Carlsson's theorem in \cite{Carlsson:SegalConjecture} was the completion theorem for $E = \SphereSpectrum_{G}$.

The importance of completion theorems is most pronounced when $E$ is a $G$-split $G$-spectrum. In this case Theorem \ref{thm:Split} gives an isomorphism $E^*_G(EG_+) \cong E^*(EG_+/G) \cong E^*(BG_+)$, where $E^*$ is the cohomology theory associated to the underlying non-equivariant spectrum of $E$, and so the completion theorem for a $G$-split spectrum $E$ can be interpreted as an isomorphism
 \[ E^*(BG_+) \cong E^*_G(\SphereSpectrum_G)^\wedge_I \, . \]
The advantage here is that one is often interested in $E^*(BG_+)$, while $E^*_G(\SphereSpectrum_G)$ is known, or easy to calculate. The tricky thing is that $I(G)$-adic completions have so far been difficult to compute in practice, but in this paper we present methods to simplify this process.

\subsection{Spectrum-level completion}
Completion theorems were given a more topological flavour by Greenlees--May in
\cite{GreenleesMay:CompBurnside}, where they defined a spectrum-level completion at ideals of the Burnside ring. This allowed them to rephrase the completion conjecture, asking, for a $G$-spectrum $E$, whether the natural map
\[ F(\SphereSpectrum_G,E) \longrightarrow F(EG_+,E), \]
induced by the map $EG \to pt$, becomes an equivalence upon $I$-adic completion. We recall some basic properties of the spectrum-level completion here and refer the reader to \cite{GreenleesMay:CompBurnside} for details.

First observe that one obtains a homotopy action of $A(G)$ on a $G$-spectrum $E$ by letting $\alpha \in A(G) \cong [\SphereSpectrum_G,\SphereSpectrum_G]^G$ act via the composite
\[ E \simeq \SphereSpectrum_G \wedge E \xrightarrow{\alpha\wedge 1} \SphereSpectrum_G \wedge E \simeq E.\]

A $G$-spectrum $T$ is \emph{$\alpha$-torsion} if the homotopy class $\alpha^n \colon T \to T$ is trivial for sufficiently large $n$. For an ideal $J \subseteq A(G)$, we say that $T$ is \emph{$J$-torsion} if $T$ is $\alpha$-torsion for all $\alpha \in J$. Define $\calT_J$ to be the class of all $J$-torsion $G$-spectra. Greenlees--May then define the $J$-adic completion $E^\wedge_J$ of a $G$-spectrum $E$ to be the Bousfield $\calT_J$-localization of $E$. 

To elaborate slightly on this, say that a spectrum $W$ is \emph{$\calT_J$-acyclic} if $W\wedge T$ is contractible for every $T \in \calT_J$. We say a spectrum $X$ is \emph{$\calT_J$-local} if $[W,X] = 0$ for every $\calT_J$-acyclic spectrum $W$, and we say that a map $f \colon E \to X$ is a \emph{$\calT_J$-equivalence} if its cofibre is $\calT_J$-acyclic. A \emph{$\calT_J$-localization} of $E$ is then a $\calT_J$-equivalence from $E$ to a $\calT_J$-local spectrum. Such a map is, up to homotopy, initial among maps from $E$ to $\calT_J$-local spectra, and it follows that $\calT_J$-localization, and hence $J$-adic completion, is unique up to homotopy. 

Greenlees--May establish the existence of $J$-completions via an explicit construction $E^\wedge_J \simeq F(M(J),E)$ for a certain spectrum $M(J)$, but also point out that existence can be established formally by adapting Bousfield's arguments from \cite{Bousfield:LocalizationOfSpectraWRTHomology} to the equivariant setting.

Greenlees--May describe the homotopy groups of a completed spectrum in terms of the zeroth and first derived functors of the $J$-adic completion functor (the higher derived functors vanish) in \cite{GreenleesMay:CompBurnside}, and these derived functors are studied in \cite{GreenleesMay:DerivedFunctors}. For finite $G$ the Burnside ring $A(G)$ is Noetherian, so when $E$ is of finite type, the first derived functor also vanishes, and there is a natural isomorphism of Mackey functors
 \[ \pi_q(E^\wedge_J) \cong  \pi_q(E)^\wedge_J \, . \]

\subsection{Non-equivariant effect of completion}

Generalizing the augmentation, we consider for a subgroup $H \leq G$ the \emph{$H$-fixed-point homomorphism}
\[ \phi^H \colon A(G) \to \Z \]
defined by setting
\[ \phi^H([X]) = |X^H| \]
for a $G$-set $X$ and extending linearly. For a subgroup $H \leq G$ and an ideal $J \subseteq A(G)$, set
\[ \phi^H(J) \defeq \left\langle  \phi^H(X)  ~ \vert X \in J  \right\rangle  \subseteq \Z \, . \] 

Greenlees proved the following useful theorem in \cite{Greenlees:EqFuncDuals}.
\begin{theorem}\cite{Greenlees:EqFuncDuals} \label{thm:Greenlees}
If $X$ is a bounded-below $G$-spectrum, then the map obtained from $X \to X^{\wedge}_{J}$ by taking $\Phi^H$-fixed points is non-equivariantly $\phi^H(J)$-adic completion.
\end{theorem}

\section{The replacement theorem} \label{sec:Replacement}

In this section we prove our first structural theorem for computing $I(G)$-adic completions, which is Theorem \ref{mthm:Replacement} of the introduction. The theorem says that when $X$ is a bounded-below $G$-spectrum, one can replace $X$ by a simpler spectrum $X[\F_\mathbf{P}]$ without changing the weak homotopy type of the $I(G)$-adic completion. The spectrum $X[\F_\mathbf{P}]$ admits a natural map $X[\F_\mathbf{P}] \xrightarrow{\iota} X$, and is characterized by the fact that $\Phi^H(\iota)$ is a non-equivariant weak equivalence if $H \in \F_\mathbf{P}$, and $\Phi^H(X[\F_\mathbf{P}])$ is weakly contractible if $H \notin \F_\mathbf{P}$. (Recall from Section \ref{sec:FamClSp} that $\F_\mathbf{P}$ is the family of subgroups in $G$ of prime-power order for any prime.) The theorem therefore says in principle that $I(G)$-adic completions are determined by prime-power fixed-point spectra.

Geometric intuition would have one think of $X[\F_\mathbf{P}]$ as the subspectrum of $X$ consisting of orbits with isotropy groups in $\F_\mathbf{P}$. This is a useful analogy but the construction does not make sense in equivariant stable homotopy theory; instead $X[\F_\mathbf{P}]$ is defined as follows.

\begin{definition} \label{def:X[F]}
Let $\F$ be a family of subgroups in $G$. For a $G$-spectrum $X$, the \emph{$\F$-isotropy restriction} of $X$ is the $G$-spectrum
 \[ X[\F]  \defeq E\F_+ \sma X \, . \] 
The functor $R_\F \colon X \mapsto X[\F]$ is the \emph{$\F$-isotropy restriction functor}. Denote by $\eta \colon R_\F \Rightarrow Id$ the natural map induced by the projection $E\F \to pt$.
\end{definition}

The natural transformation map $\eta$ is induced by a trivial map, and to understand its proper role we make the following remark. Recall from Section \ref{sec:FamClSp} that $\F_{all}$, the family consisting of all subgroups of $G$ has a universal space homotopy equivalent to a point. For a $G$-spectrum $X$ we therefore have a homotopy equivalence   $X[\F_{all}] \xrightarrow{\simeq}   X .$ The inclusion $\F \subseteq \F_{all}$ induces an inclusion $E\F \hookrightarrow E\F_{all}$ and hence a map $X[\F] \to X[\F_{all}]$, which we compose with the equivalence $X[\F_{all}] \xrightarrow{\simeq}   X $ to get $\eta_X$. Thus the map $E\F \to pt$ in the definition above should be thought of as the composite $E\F \to E\F_{all} \xrightarrow{\simeq} pt$, but in the end projection to a point is trivial however one factors it.

The next lemma records some elementary properties of the $\F$-isotropy restriction. 

\begin{lemma}
Let $\F$ be a family of subgroups in $G$.
\begin{enumerate}
 \item[(a)] $R_\F$ preserves the properties of being bounded below, of finite type or a $G$-CW spectrum.
 \item[(b)] $R_\F$ commutes with colimits.
 \item[(c)] Let $X$ be a $G$-spectrum, and let $H$ be a subgroup of $G$. If $H\notin \F$ then $\Phi^H(X[\F])$ is non-equivariantly weakly contractible. If  $H\in \F$ then $\Phi^H(X[\F]) \xrightarrow{\Phi^H(\eta)} \Phi^H(X)$ is a non-equivariant equivalence.
 \item[(d)] An inclusion of families $\F_a \subseteq \F_b$ induces a natural transformation $R_{\F_a} \Rightarrow R_{\F_b}$ in which every map is a cofibration.
\end{enumerate}
\end{lemma}
\begin{proof}
Part (a) is immediate since $E\F$ is a $G$-CW complex of finite type, and (b) is a direct consequence of the smash product definition of $R_\F$.

Since $\Phi^H$ commutes with suspension and smash product, one has a non-equivariant equivalence, 
\[ \Phi^H(X[\F]) \simeq (E\F_+)^H \sma \Phi^H(X) \, . \]
Now, $(E\F_+)^H$ is contractible for $H \notin \F$, and equivalent to $S^0$  for $H \in \F$. Part (c) follows.

The natural transformation in part (d) is induced by the cofibration $E\F_a \to E\F_b$ in the obvious way.
\end{proof}

A less obvious property of isotropy restriction is that it commutes with completion up to weak equivalence. To make sense of this, we first need a map comparing the two ways of composing the isotropy restriction and completion functors. Let $X$ be a $G$-spectrum. For a family of subgroups $\F$ and an ideal $J \subseteq A(G)$, a commutative diagram 
\[ 
\xymatrix{ & X^\wedge_J[\F] \ar[dd] \\
           X[\F] \ar[ur] \ar[dr]& \\ 
           & X[\F]^\wedge_J}
\]
is obtained as the diagram
\[ 
\xymatrix{ & F(\SphereSpectrum_G,X)\sma F(\SphereSpectrum_G,E\F_+) \ar[r] & F(M(I),X) \sma F(\SphereSpectrum_G,E\F_+) \ar[dd]^h \\
  X \sma E\F_+  \ar[ur]^{\cong} \ar[dr]^{\cong} &&\\
  & F(\SphereSpectrum_G \sma \SphereSpectrum_G, X \sma E\F_+) \ar[r] & F(M(I)\sma \SphereSpectrum_G, X\sma E\F_+) \, ,
}
\]
where the horizontal maps are each induced by the natural map $M(I) \to \SphereSpectrum_G$, and $h$ is the canonical map.

\begin{lemma} \label{lem:IsoResCompCommute}
Let $\F$ be a family of subgroups of $G$ and let $J$ be an ideal of $A(G)$. If $X$ is a bounded-below $G$-spectrum, then the natural map $h \colon X^\wedge_J[\F] \to X[\F]^\wedge_J$ is a weak equivalence.
\end{lemma}

\begin{proof}
By Proposition \ref{prop:WeakGeometric} it suffices to show that $\Phi^H(f)$ is a non-equivariant weak equivalence for every subgroup $H$ of $G$. 

When $H \notin \F$, the spectrum $\Phi^H(X^\wedge_J[\F])$ is weakly contractible. By Theorem \ref{thm:Greenlees}, $\Phi^H( X[\F]^\wedge_J )$ is non-equivariantly equivalent to the $\phi^H(J)$-adic completion of the weakly contractible spectrum $\Phi^H( X[\F])$ and is therefore itself weakly contractible. It follows trivially that $\Phi^H(h)$ is a weak equivalence.

For $H \in \F$, Theorem \ref{thm:Greenlees} gives non-equivariant equivalences 
\[ \Phi^H(X^\wedge_J[\F]) \simeq \Phi^H(X^\wedge_J) \simeq \Phi^H(X)^\wedge_{\phi^H(J)} \,   \]
and
\[ \Phi^H(X[\F]^\wedge_J) \simeq \Phi^H(X[\F])^\wedge_{\phi^H(J)} \simeq \Phi^H(X)^\wedge_{\phi^H(J)} \, . \]
In fact the theorem implies that the natural maps $\Phi^H(X[\F]) \to  \Phi^H(X^\wedge_J[\F])$ and $\Phi^H(X[\F]) \to  \Phi^H(X[\F]^\wedge_J)$ are non-equivariantly $\phi^H(J)$-adic completion maps. Since $\Phi^H(h)$ respects these maps, it follows that it is a non-equivariant equivalence. 
\end{proof}

Having defined the replacement spectra, we now prepare to prove the Replacement Theorem. Proposition \ref{prop:WeakGeometric} and Theorem \ref{thm:Greenlees} allow us to control the $I(G)$-adic completion of a spectrum if we can describe the ideals  $\phi^H(I(G))$. This is done in the following lemma.

\begin{lemma} \label{lem:nH}
For a subgroup $H \leq G$ we have
 \[ \phi^H(I(G)) = \begin{cases}  (0), ~ &\text{if}~ H = \{1\}; \\
                         (p^k)~\text{for some}~ k>0,   &\text{if}~ H~ \text{is a $p$-group; and}  \\
 												 \Z,~  &\text{otherwise.} 
                                   
\end{cases} \]
\end{lemma}
\begin{proof}
The first claim, $\phi^1(I(G)) = \epsilon(I(G)) = 0$, is true by definition of $I(G)$.

For nontrivial $H$ we observe that the ideal $I(G)$ is generated by elements of the form 
\[ [G/K] - |G:K| \cdot [G/G]\] 
where $K$ runs through subgroups of $G$, and thus $\phi^H(I(G))$ is generated by the numbers 
 \[(G/K)^H - |G:K| = \frac{|N_G(H,K)|}{|K|} - |G:K|\]
where the transporter $N_G(H,K)$ is the set of elements $x \in G$ such that $xHx^{-1} \leq K$.

For a prime $q$, let $G_q$ be a Sylow $q$-subgroup of $G$. Observe that if $H$ is not a $q$-group then $N_G(H,G_q)$ is empty so 
 \[(G/G_q)^H - |G:G_q| = - |G:G_q| \, . \]
Suppose that the order of $H$ is divisible by two distinct primes. Then $H$ is not a $q$-group for any prime $q$, so, by the previous observation, $|G:G_q| \in \phi^H(I(G))$ for every prime $q$ dividing $G$. As the greatest common divisor of the numbers $|G/G_q|$ is $1$, this implies that $\phi^H(I(G)) = \Z$.

Now suppose that $H$ is a nontrivial $p$-group. An argument similar to the previous paragraph shows that $|G_p|$, which is the greatest common divisor of the collection of numbers $|G:G_q|$ for primes $q \neq p$, is contained in $\phi^H(I(G))$. Therefore the ideal $\phi^H(I(G))$ contains a power of $p$. On the other hand, for every finite $G$-set $X$ the set $X \setminus X^H$ is a disjoint union of nontrivial $H$-orbits. As $H$ is a  $p$-group each of these orbits has cardinality divisible by $p$. Since the numbers $|X^H| -|X| = - |X \setminus X^H|$ generate $\phi^H(I(G))$, this implies that $\phi^H(I(G))$ is contained in the ideal $(p)$. Thus $\phi^H(I(G))$ is of the form $(p^k)$ for some $k>0$.
\end{proof}

We can now prove the replacement theorem. 
\begin{proof}[Proof of Theorem \ref{mthm:Replacement}]
For every subgroup $H$ of $G$, the natural map $f \colon X[\F_\mathbf{P}] \to X$ induces a non-equivariant commutative diagram
\[ 
\xymatrix{
\Phi^H(X[\F_\mathbf{P}]) \ar[rr]^{\Phi^H(f)} \ar[dd] && \Phi^H(X) \ar[dd] \\ 
\\ 
{\Phi^H(X[\F_\mathbf{P}]_{I}^{\wedge})} \ar[rr]^{\Phi^H(f^{\wedge}_{I(G)})} && {\Phi^H(X_{I}^{\wedge})},
}
\]
where, by Theorem \ref{thm:Greenlees}, the two vertical maps are $\phi^H(I(G))$-adic completion maps, and the bottom map is the $\phi^H(I(G))$-adic completion of the top map. When $H$ is a group of prime power order the top map is a weak equivalence and so the same is true of the bottom map. When $H$ has order divisible by two primes the two vertical maps are completion at $\phi^H(I(G)) = \Z$, so both spectra in the bottom row are weakly contractible and the map between them is a weak equivalence by default. Since $\Phi^H(f^{\wedge}_{I(G)})$ is a weak equivalence for every subgroup $H$ of $G$ it follows from Proposition \ref{prop:WeakGeometric} that $f^{\wedge}_{I(G)}$ is a weak equivalence.
\end{proof}

\section{The decomposition theorem} \label{sec:Decomposition}
In this section we explain the importance of Theorem \ref{mthm:Replacement} by showing that the $I(G)$-adic completion of the replacement spectrum $X[\F_\mathbf{P}]$ can be computed in an efficient manner. That is, we show that the $I(G)$-adic completion of $X[\F_\mathbf{P}]$ decomposes as a colimit of $I(G)$-adic completions, the components of which can be computed as the completion at a prime. The first step in this direction is the following observation.

\begin{lemma} \label{lem:X[F]Colimit}
For a $G$-spectrum $X$, the natural map
 \[ \colim_{\F \in \calD} X[\F] \longrightarrow X[\F_\mathbf{P}] \] 
is a homotopy equivalence.
\end{lemma}
\begin{proof}
This follows from Lemma \ref{lem:EFColimit} and the fact that smash products preserve colimits. 
\end{proof}

With a bit more work we can prove the following proposition, of which Theorem \ref{mthm:Decomposition} in the introduction is a special case.

\begin{proposition} \label{prop:X[F]CompColimit}
Let $J$ be an ideal of $A(G)$. For a bounded-below $G$-spectrum $X$, the natural map
\[ \hocolim_{\F \in \calD} \left( X[\F]^\wedge_{J} \right) \longrightarrow X[\F_\mathbf{P}]^\wedge_J \]
is a weak equivalence. 
\end{proposition}
\begin{proof}
This map fits into a commutative diagram
\[
 \xymatrix{ 
 \hocolim\limits_{\F \in \calD} \left( X^\wedge_{J}[\F]\right) \ar[d]^\simeq_{(\ref{lem:IsoResCompCommute})} \ar[r]^{ \simeq} & 
 \colim\limits_{\F \in \calD} \left( X^\wedge_{J}[\F]\right)  
 \ar[r]^{\quad \simeq}_{\quad (\ref{lem:X[F]Colimit})} & X^\wedge_J[\F_\mathbf{P}] \ar[d]^\simeq_{(\ref{lem:IsoResCompCommute})} \\
 \hocolim\limits_{\F \in \calD} \left( X[\F]^\wedge_{J} \right)  \ar[rr] && X[\F_\mathbf{P}]^\wedge_J \, .
 }
\]
As indicated, the maps in the upper right corner are equivalences by Lemmas \ref{lem:X[F]Colimit} and \ref{lem:IsoResCompCommute}, respectively. Lemma \ref{lem:IsoResCompCommute} also implies for each $\F \in \calD$, the map  $X[\F]^\wedge_{J} \to X^\wedge_{J}[\F]$ is a weak equivalence, and hence the induced map of homotopy colimits on the left side is a weak equivalence. The map from the homotopy colimit to the colimit along the top is an equivalence since each map in $X[\F_1] \to X[\F_p]$ in the diagram over which the colimit is taken is a cofibration (Lemma \ref{lem:EFnice}). It follows that the bottom map is a weak equivalence.
\end{proof}

Again we observe that in Lemma \ref{lem:X[F]Colimit} and Proposition \ref{prop:X[F]CompColimit}, the diagram $\calD$ (defined in Section \ref{sec:FamClSp}) can be restricted to just include families $\F_p$ for primes $p$ that divide the order of $G$ and the results still hold true. This is obviously more convenient for computational purposes, since one then obtains a finite diagram. However, by including all primes in the diagram, the result is natural in $G$ when $X$ is.

For a bounded-below $G$-spectrum $X$, Theorem \ref{mthm:Replacement} and Proposition \ref{prop:X[F]CompColimit} combine to give a weak equivalence
 \[  \hocolim_{\F \in \calD} \left( X[\F]^\wedge_{I(G)} \right) \simeq X^\wedge_{I(G)} \, . \] 
We now show that the completions inside the homotopy colimit can be computed in terms of $p$-adic completions. That is, for a prime $p$, the $I(G)$-adic completion of the spectrum $X[\F_p]$ is very close to being a $p$-adic completion. The difference lies only in the cofibration $X[\F_1] \to X[\F_p]$, the spectrum $X[\F_1]$ being unchanged by $I(G)$-adic completion. To make this precise, we make the following definition.

\begin{definition}
Let $\F_a \subseteq \F_b$ be families of subgroups in $G$. For a $G$-spectrum $X$, denote by $X[\F_b,\F_a]$ the cofibre of the map $X[\F_a] \to X[\F_b]$.
\end{definition}

Theorem \ref{mthm:X[F]Comp} in the introduction is included in the following proposition as parts (a) and (c).

\begin{proposition} \label{prop:X[F]Comp}
Let $X$ be a bounded-below $G$-spectrum.
\begin{itemize}
  \item[(a)] The map $X[\F_1] \to X[\F_1]^{\wedge}_{I(G)}$ is a weak equivalence.
  \item[(b)] For a prime $p$, the map $X[\F_p,\F_1] \to X[\F_p,\F_1]^{\wedge}_{I(G)} $ is, up to weak equivalence, a $p$-adic completion map.
  \item[(c)] For a prime $p$, the completion $X[\F_p]^{\wedge}_{I(G)}$ fits, up to weak equivalence, into a cofibration sequence
\[ X[\F_1] \longrightarrow X[\F_p]^{\wedge}_{I(G)} \longrightarrow \pComp{X[\F_p,\F_1]} \, . \]
\end{itemize}
\end{proposition}
\begin{proof}
We prove part (a) by showing that the induced maps on geometric fixed-point spectra are non-equivariant weak equivalences for all subgroups $H \leq G$ and then applying Proposition \ref{prop:WeakGeometric}. For a nontrivial subgroup $H \leq G$, the map $\Phi^H(X[\F_1]) \to \Phi^H(X[\F_1]^{\wedge}_{I(G)})$ is a map of weakly contractible spectra, and hence a weak equivalence. For the trivial subgroup, the map $\Phi^1(X) \to \Phi^1(X[\F_1]^{\wedge}_{I(G)})$ is non-equivariantly completion at the ideal $n_1(I(G)) = 0$, which has no effect, so the map is a weak equivalence. 

Turning to part (b), we first note that for a subgroup $H \leq G$ we have, up to weak equivalence, a cofibration sequence
 \[ \Phi^H(X[\F_1]) \longrightarrow   \Phi^H(X[\F_p]) \longrightarrow   \Phi^H(X[\F_p,\F_1]) \, . \]
We deduce that when $H \notin \F_p$ or $H = 1$, the geometric fixed-point spectrum $\Phi^H(X[\F_p,\F_1])$ is weakly contractible. Hence the map $\Phi^H(X[\F_p,\F_1]) \to \Phi^H(X[\F_p,\F_1]^{\wedge}_{I(G)}) $ can trivially be regarded as a $p$-adic completion map. For $H \in \F_p \setminus \F_1$, the map $\Phi^H(X[\F_p,\F_1]) \to \Phi^H(X[\F_p,\F_1]^{\wedge}_{I(G)}) $ is non-equivariantly completion at $(p)$ by Theorem \ref{thm:Greenlees} and Lemma \ref{lem:nH}. To complete the proof we consider the commutative square
\begin{equation} \label{eq:FpF1Square}
\xymatrix{
  X[\F_p,\F_1]  \ar[rr]^{(-)^{\wedge}_{I(G)}} \ar[d]^{(-)^{\wedge}_{p}} && X[\F_p,\F_1]^{\wedge}_{I(G)} \ar[d]^{(-)^{\wedge}_{I(G)+(p)}} \\
  X[\F_p,\F_1]^{\wedge}_{p} \ar[rr]^{(-)^{\wedge}_{I(G)+(p)}}       && X[\F_p,\F_1]^{\wedge}_{I(G)+(p)}
\, .} 
\end{equation}
The preceding argument shows that taking geometric fixed points at a subgroup $H \leq G$ of this  square one obtains the commutative square 
\[ 
\xymatrix{
  \Phi^H(X[\F_p,\F_1])  \ar[rr]^{(-)^{\wedge}_{p}} \ar[d]^{(-)^{\wedge}_{p}} && \Phi^H(X[\F_p,\F_1]^{\wedge}_{I(G)}) \ar[d]^{\simeq} \\
  \Phi^H(X[\F_p,\F_1]^{\wedge}_{p}) \ar[rr]^{\simeq}       && \Phi^H(X[\F_p,\F_1]^{\wedge}_{I(G)+(p)})
\, .} 
\]
From this we deduce that in the square \eqref{eq:FpF1Square}, the lower and right side maps are weak equivalences, and thus the upper map is equivalent to the left map, up to weak equivalence, which is what we wanted to prove.

Part (c) is an easy consequence of parts (a) and (b), using the fact that completions preserve cofibrations up to weak equivalences (they preserve fibration sequences by their construction as mapping spectra, and cofibration sequences are fibration sequences up to weak equivalence).  
\end{proof}

As promised, the combination of Theorems \ref{mthm:Replacement}, \ref{mthm:Decomposition} and \ref{mthm:X[F]Comp} provides a method for computing $I(G)$-adic completions of bounded-below $G$-spectra. The overall procedure may seem slightly complicated, so we permit ourselves a small departure from precision to say that in essence these results allow us to compute $X^\wedge_{I(G)}$ as a wedge sum $\bigvee_p \pComp{X[\F_p]} $. Enforcing precision again, things are a bit more complicated since the spectra $X[\F_p]$ are joined together at $X[\F_1]$, necessitating the homotopy colimit construction instead of a simple wedge sum, and $X[\F_1]$ is unaffected by completion, making the completions slightly harder to describe. Some instances where the wedge sum description can be made precise are given in the following corollary.

\begin{corollary} \label{cor:CompAsWedge}
Let $X$ be a $G$-spectrum that is bounded below.
\begin{enumerate}
\item[(a)] If $X[\F_1]$ splits as a wedge sum of its $p$-completions, at different primes $p$, then there is a weak equivalence
 \[ X^\wedge_{I(G)} \simeq \bigvee_p \pComp{X[\F_p]} \, .  \]
\item[(b)] If the cofibration $X[\F_1] \to X$ splits, then there is a weak equivalence
 \[ X^\wedge_{I(G)} \simeq X[\F_1] \vee \bigvee_p \pComp{X[\F_p,\F_1]} \, . \]
\end{enumerate}
\end{corollary}

Part (a) in Corollary \ref{cor:CompAsWedge} is particularly useful since it applies when $X[\F_1]$ is a torsion spectrum, such as the classifying space of a group. Part (b) is useful to deal with basepoint issues when $X$ is the suspension spectrum of a $G$-space with added, disjoint basepoint. Upon suspension, the added basepoint becomes a sphere spectrum wedge summand that belongs to $X[\F_1]$ and is unaffected by completion. In some cases the result allows us to set this sphere summand aside while completing and add it back at the end.

\section{Stable maps between classifying spaces} \label{sec:StableMaps}

Let $G$ be a finite group. For a compact Lie group $K$, a \emph{principal $(G,K)$-bundle} is a principal $K$-bundle in the category of $G$-spaces. Denote by $B(G,K)$ the classifying space of principal $(G,K)$-bundles. Generalizing the Segal conjecture, May--Snaith--Zelewski proved the completion theorem for $\Stable_G B(G,K)_+$ in \cite{MaySnaithZelewski:FurtherGeneralizationOfSegal}. (This was originally proved for finite $K$ in \cite{LewisMayMcClure:Segal}.) Thus the natural map 
 \[ \Stable_G B(G,K)_+ \cong F(S^0,  \Stable_G B(G,K)_+ ) \longrightarrow F(EG_+, \Stable_G B(G,K)_+) \]
is an $I(G)$-adic completion map.

Now, $\Stable_G B(G,K)_+$ is $G$-split with underlying non-equivariant spectrum $BK$ so by Theorem \ref{thm:Split} we have a non-equivariant equivalence
\[ F(EG_+,\Stable_G B(G,K)_+)^G \simeq F(EG/G_+, \iota^*(\Stable_G B(G,K)_+)) \simeq F(BG_+,\Stable BK_+) \, . \]
Thus the spectrum of stable maps between classifying spaces can be computed as the completion of $ \Stable_G B(G,K)_+^G$. The tom Dieck splitting of this spectrum was further refined in \cite{LewisMayMcClure:Segal} to obtain the non-equivariant splitting,
 \begin{equation} \label{eq:B(G,K)splitting} 
 \Stable_G B(G,K)_+^G \simeq \bigvee_{[H,\varphi]} \Stable BW(H,\varphi)_+ \, .
 \end{equation}  
Here the wedge sum ranges over conjugacy classes of pairs $(H,\varphi)$ where $H \leq G$ and $\varphi$ is a group homomorphism $H \to K$, and $W(H,\varphi)$ is the Weyl group
 \[ W(H,\varphi) = N_{G\times K}(\Delta(H,\varphi))/ \Delta(H,\varphi), \]
of the subgroup
 \[ \Delta(H,\varphi) = \{ (h,\varphi(h)) \mid h \in H \} \leq G\times K \, .\]
(Equivalently, $W(H,\varphi)$ is the automorphism group of the principal $(G,K)$-bundle \mbox{$(G\times K)/\Delta(H,\varphi) \to G/H$}.)

Collecting these facts, topologists commonly say that $F(BG_+,\Stable BK_+)$ is the $I(G)$-adic completion of the wedge sum of classifying spaces in \eqref{eq:B(G,K)splitting}. This statement has up to now been rather vague and of limited practical use as the action of $A(G)$ in terms of this non-equivariant splitting is not understood, and the effect of $I$-adic completion has so far been a mystery. Using the results from Sections \ref{sec:Replacement} and \ref{sec:Decomposition} we can now make this a precise statement.

First we note that, for a family $\F$ of subgroups in $G$, we have 
 \[ \Stable_G B(G,K)_+[\F] \simeq \Stable_G B(\F,K)_+ \, ,   \]
where $B(\F,K)$ is the classifying space of principal $K$-bundles in the category of $\F$-spaces. The fixed-point spectrum $\Stable_G B(\F,K)_+^G$ admits a non-equivariant splitting similar to the one for $ \Stable B(G,K)_+^G $. 
\begin{proposition}[\cite{LewisMayMcClure:Segal}] \label{prop:Splitting} 
Let $G$ be a finite group and $K$ a compact Lie group. For a family $\F$ of subgroups in $G$, there is a non-equivariant splitting
\[ \Stable B(\F,K)_+^G \simeq   \bigvee_{\substack{[H,\varphi] \\ H \in \F}} \Stable{BW(H,\varphi)}_+ \, .\]
\end{proposition}
\begin{proof}
Lewis--May--McClure proved this for $B(G,K) = B(\F_{all},K$) in \cite{LewisMayMcClure:Segal}. A minor adjustment generalizes their proof to any family of subgroups.
\end{proof}

This leads to the following description of stable maps between classifying spaces, which is Theorem \ref{mthm:StableMaps} of the introduction, and generalizes \cite[Theorem B]{KR:pSegal}.
\begin{theorem} \label{thm:StableMaps}
For a finite group $G$ and a compact Lie group $K$, there is a homotopy equivalence
\[ F(BG_+, \Stable BK_+) 
    \simeq  \bigvee_{\substack{[H,\varphi] \\ H \in \F_\mathbf{P}}} 
    \Stable (BW(H,\varphi)^{\wedge}_{p(H)} )_+ \, , 
\] 
where, for $H \in \F_\mathbf{P}\setminus \F_1$, $p(H)$ is the prime that divides $|H|$, and $p(1) = 0$.
\end{theorem}

\begin{proof}
Proposition \ref{prop:Splitting} gives us 
\[ \Stable B(\F_1,K)_+^G \simeq \Stable{BW(1,i)}_+ \, \]
where $i$ is the unique homomorphism $1 \to K$, and 
 \[  \Stable B(\F_p,K)_+^G \simeq   \bigvee_{\substack{[P,\varphi] \\ P \in \F_p}} \Stable{BW(P,\varphi)}_+ \, .\]
Thus $\Stable B(\F_1,K)_+^G$ is a summand of $\Stable B(\F_p,K)_+^G $ with cofibre
 \[  \bigvee_{\substack{[P,\varphi] \\ 1 < P \in \F_p}} \Stable{BW(P,\varphi)}_+ \,.\]
Applying part (b) of Corollary \ref{cor:CompAsWedge} we therefore have a weak equivalence 
 \[ F(BG_+ , \Stable BK_+) \simeq  \Stable{BW(1,i)}_+ \vee \bigvee_p \bigvee_{\substack{[P,\varphi] \\ 1 < P \in \F_p}} (BW(P,\varphi)^{\wedge}_{p} )_+ \, , \] 
and this is just a different way to write the wedge sum in the statement of the theorem.

What remains is to note that (by Morse theory) $K$ has the homotopy type of a CW-complex. Therefore $B(G,K)$ also has the homotopy type of a CW-complex, and all the spectra in sight have the homotopy type of $G$-CW spectra. Hence the weak equivalence is in fact a homotopy equivalence.
\end{proof}

We pause to note that
 \[  BW(1,i)_+ = B(G\times K)_+ \simeq BG_+ \sma BK_+ \, .\]
We have a stable splitting
\[ \Stable BG_+ \simeq \SphereSpectrum \vee \bigvee_p \Stable \pComp{BG} \, ,\]
where the sphere spectrum summand arises from the natural equivalence $\Stable X_+ \simeq \SphereSpectrum \vee \Stable X$, and the wedge sum can be taken over all primes $p$ or just those primes that divide the order of $G$. It follows that we have a splitting
 \[ 
 \Stable B(\F_1,K)_+^G  \simeq \Stable BK_+ \vee \bigvee_p \left( \Stable \pComp{BG} \sma \Stable BK_+ \right) \, ,
 \] 
and we can rewrite the conclusion of Theorem \ref{thm:StableMaps} as a homotopy equivalence
\[ 
  F(BG_+,\Stable BK_+) \simeq 
  \Stable{BK_+} \vee
  \bigvee_p \left( \left( \Stable \pComp{BG} \sma \Stable BK_+ \right) 
  \vee \bigvee_{\substack{[P,\varphi] \\ 1 < P \in \F_p}}   
  (BW(P,\varphi)^{\wedge}_{p} )_+   \right) \, .
\] 
The difference here is just a matter of bookkeeping, but it allows us to easily focus attention on one prime at a time, obtaining the following corollary, which generalizes \cite[Theorem A]{KR:pSegal}.

\begin{corollary} \label{cor:pCompStableMaps}
For a finite group $G$, a compact Lie group $K$ and a prime $p$, there is a homotopy equivalence
 \[ F( \pComp{BG}_+ , BK_+) 
    \simeq \left( \Stable \pComp{BG}_+ \sma \Stable BK_+ \right) 
    \vee \bigvee_{\substack{[P,\varphi] \\ 1 < P \in \F_p}}           
         (BW(P,\varphi)^{\wedge}_{p} )_+ \, . 
 \]
\end{corollary}
\begin{proof}
For a prime $q$, let $S_q$ be a Sylow $q$-subgroup of $G$. Standard arguments show that ${{BG}^{\wedge}_q}{}_+$ is a stable summand of both $BG_+$ and ${BS_q}{}_+$. Hence $F( {{BG}^{\wedge}_q}{}_+ , BK_+)$ is a stable summand of both $F( BG_+ , BK_+)$ and 
$F( {BS_q}{}_+ , BK_+)$. Furthermore, since $BG$ splits stably as a wedge sum of its $q$-completions, every summand of $F( BG_+ , BK_+)$ is a summand of $F( {BS_q}{}_+ , BK_+)$ for some prime $q$. Considering which summands of  $F( BG_+ , BK_+)$ can belong to which $F( {BS_q}{}_+ , BK_+)$ forces the desired result. For these purposes, the leading term in the wedge sum is rewritten as 
 \[  \Stable \pComp{BG}{}_+ \sma \Stable BK_+ 
     \simeq (\SphereSpectrum \vee \Stable \pComp{BG} ) \sma \Stable BK_+ 
     \simeq \Stable BK_+ \vee  ( \pComp{BG} \sma \Stable BK_+) \, .
 \]
\end{proof}

A case of special interest is when $K = 1$ and Theorem \ref{thm:StableMaps} and Corollary \ref{cor:pCompStableMaps} describe the stable mapping duals of $BG_+$ and $\pComp{BG}{}_+$, respectively.

\begin{corollary} \label{cor:MappingDuals}
For a finite group $G$, there are non-equivariant equivalences
\[ D(BG_+) \simeq \bigvee_{H \in \F_\mathbf{P}} \Stable(BW_G(H)^\wedge_{p(H)})_+, \]
and
\[ D(\pComp{BG}{}_+) \simeq \bigvee_{H \in \F_p} \Stable (\pComp{BW_G(H)})_+ \, . \]
\end{corollary}

\bibliographystyle{acm}
\bibliography{pStrongSegal}
\end{document}